\pdfoutput=1

\documentclass[twoside,lm,fleqn,final]{pre}

\usepackage[applemac]{inputenc}
\usepackage{graphicx,math}

\let\prece\preccurlyeq
\let\tht\theta
\let\phi\varphi
\let\Gm\varGamma

\def\Bt{\tilde B}
\def\nt{\tilde n}
\def\vt{\tilde v}
\def\wt{\tilde w}

\def\o{\text{\upshape o}}
\def\Hh{\mathcal{H}}
\def\Ho{\Hh^\o}
\def\Bo{B^\o}
\def\Do{D^\o}
\def\Ali{\smash{A_\lm\inv}}
\def\VAli{V\!\Ali}
\def\Tin{\smash{\hat T_n}}
\def\spn{\operatorname{span}}

\def\w#1{\langle#1\rangle}
\def\nw#1{\nn{#1}_w}
\def\no#1{\nn{#1}_\o}
\def\eval#1\at{\rbar{#1}}

\newtheorem{thm}     {Theorem}

\newtheorem{lem}     {Lemma}
\newtheorem{prp}[lem]{Proposition}

\begin{document}  

\date
   {Version 4.3 August 2004}
\title 
   {Hill's Potentials in\\Weighted Sobolev Spaces and their\\Spectral Gaps}
\author 
   {Jürgen Pöschel}
\address
   {Institut für Analysis, Dynamik und Optimierung, 
    Universität Stuttgart\\
    Pfaffenwaldring 57, D-70569 Stuttgart\\
    poschel@mathematik.uni-stuttgart.de, or j@poschel.de}
\maketitle

\section{Results}

In this paper we consider the Schrödinger operator
\[
  L = -\frac{\d^2}{\dx^2}+q
\]
on the interval $[0,1]$, depending on an $L^2$-potential $q$ and endowed with periodic or anti-periodic boundary conditions. In this case, $L$ is also known as \emph{Hill's operator}. Its spectrum is pure point, and for real $q$ consists of an unbounded sequence of real \emph{periodic eigenvalues}
\[
  \lm_0^+(q)<\lm_1^-(q)\le\lm_1^+(q)<
  \dots<\lm_n^-(q)\le\lm_n^+(q)<\cdots\enspace.
\]
Their asymptotic behaviour is
\[
  \lm_n^\pm = n^2\pi^2 + \mean*{q} + \ell^2(n),
\]
where $\mean*{q}$ denotes the mean value of $q$. Equality may occur in every place with a ‘$\le$’-sign, and one speaks of the \emph{gap lengths}
\[
  \gm_n(q) = \lm_n^+(q)-\lm_n^-(q), \qquad n\ge1,
\]
of the {potential} $q$. If a gap length is zero, one speaks of a \emph{collapsed gap}, otherwise of an \emph{open gap}.

We recall that the gaps separate the \emph{spectral bands}
\[
  B_n = \bra[1]{\lm_{n-1}^+,\lm_{n}^-}, \qquad n\ge1,
\]
which are dynamically characterized as the locus of those real~$\lm$, for which all solutions of $Lf=\lm f$ are bounded. In other words, for any $\lm$ in the interior of an open gap as well as for all $\lm<\lm_0^+$, any nontrivial solution of $Lf=\lm f$ is unbounded.

For complex $q$, the periodic eigenvalues are still well defined, but in general not real, since $L$ is no longer self-adjoint. Their asymptotic behaviour is the same, however, and we may order them lexicographically -- first by their real part, then by their imaginary part -- so that
\[
  \lm_0(q)\prec\lm_1^-(q)\prece\lm_1^+(q)\prec\dots
  \prec\lm_n^-(q)\prece\lm_n^+(q)\prec\cdots.
\]
The gap lengths are then defined as before, but may now be complex valued. They are also no longer characterized dynamically.

We are interested in the relationship between the regularity of a potential and the sequence of its gap lengths. Mar\v{c}enko \& Ostrowsk\u{\i} \cite{MO} showed that
\[
  q\in H^k(S^1,\RR)
  \equi {\sum_{n\ge1}} n^{2k}\gm_n^2(q)<\infty
\]
for all nonnegative integers~$k$, while Hochstadt~\cite{Ho} even earlier observed that
\[
  q\in C^\infty(S^1,\RR) 
  \equi \gm_n(q) = O(n^{-k})\enspace\text{for all $k\ge0$}.
\]
Trubowitz \cite{Tr} then proved that
\[
  q\in C^\om(S^1,\RR) 
  \equi \gm_n(q) = O(\e^{-an})\enspace\text{for some $a>0$}.
\] 
Later, due to the realization of the periodic KdV flow as an isospectral deformation of Hill's operator, other regularity classes such as Gevrey functions were also taken into account, as well as non-real potentials. Recent results in this direction are for example due to Sansuc \&~Tkachenko~\cite{ST-96}, Kappeler \&~Mityagin \cite{KM-99,KM-01} and Djakov \& Mityagin \cite{DM-02,DM-03}.
All this shows that within certain limits, one may think of the gap lengths as another kind of Fourier coefficients of the potential.

It is the purpose of this paper to further extend these results and to give a new, short, self-contained proof that applies simultaneously to all cases. This proof does not employ any conformal mappings, trace formula, asymptotic expansions, iterative arguments, or other convolutions. Instead, the essential ingredient is the inverse function theorem.

To set the stage, we introduce \emph{weighted Sobolev spaces} $\Hh^w$ as follows \cite{KM-99,KM-01}. A \emph{normalized weight} is a function $w\maps\ZZ\to\RR$ with
\[
 w_n = w_{-n} \ge 1
\]
for all $n$, and the class of all such weights is denoted by~$\Ww$. The \m{w}-norm $\nw{q}$ of a complex 1-periodic function
$
  q = \sum_{n\in\ZZ} q_n\e^{2n\pi ix}
$
is then defined through
\[
  \nn{q}^2_{w} = \sum_{n\in\ZZ} w_n^2\n{q_n}^2,
\]
and 
\[
  \Hh^w = \set*{q\in L^2(S^1,\CC): \nw{q}<\infty}
\]
is the Banach space of all such functions with finite $w$-norm. Note that
\[
  \Ho \defequal \bigcup_{w\in\Ww} \Hh^w = L^2(S^1,\CC),
\]
since all weights are assumed to be at least~1.

Here are some examples of relevant weights. The trivial weights
$
  w_n = 1
$
give rise to the underlying Banach space $\Ho=L^2(S^1)$.
Letting $\w{n}=1+\n{n}$ and $r\ge0$, $a\ge0$, the polynomial weights
\[
  w_n = \w{n}^r
\] 
give rise to the usual Sobolev spaces $H^r(S^1)$, and the exponential weights
\[
  w_n = \w{n}^r\e^{a\n{n}}
\] 
give rise to spaces $H^{r,a}(S^1)$ of functions in $L^2(S^1)$, that are analytic on the strip $\n{\Im z}<a/2\pi$ with traces in $H^r(S^1)$ on the boundary lines. In between are, among others, the subexponential weights
\[
  w_n = \w{n}^r\e^{a\n{n}^{\sg}}, \qquad 0<\sg<1,
\] 
giving rise to Gevrey spaces $H^{r,a,\sg}(S^1)$, and weights of the form
\[
  w_n = \w{n}^r\exp\pasii{\frac{a\n{n}}{1+\log^{\,\al}\!\w{n}}}, \qquad \al>0.
\] 
More examples are given below.

For the most part we will be concerned with the subclass $\Mm\subset\Ww$ of weights that are also \emph{submultiplicative}. That is,
\[
  w_{n+m}\le w_nw_m
\]
for all $n$ and~$m$. This implies in particular that $w_n\le w_1^n$ for all $n\ge1$, so submultiplicative weights can not grow faster than exponentially. All the weights given above are submultiplicative, and
\[
  \Hh^\om \defequal \bigcap_{w\in\Mm} \Hh^w 
\]
is the space of all entire functions of period~1. It turns out that only in the submultiplicative case, and more precisely in the subexponential case, there is a one-to-one relationship between the decay rates of Fourier coefficients and spectral gap lengths.

We begin by considering the forward problem of controlling the gap lengths of a potential in terms of its regularity, first for submultiplicative weights -- see \cite{KM-01}.

\begin{thm} \label{t:gaps}
If $q\in\Hh^w$ with $w\in\Mm$, then
\[
  \sum_{n\ge N} w_{n}^2\n{\gm_n(q)}^2 \le
  9\nw{T_Nq}^2+\frac{576}{N}\nw{q}^4
\]
for all $N\ge4\nw{q}$, where $T_Nq = \sum_{\n{n}\ge N} q_n\e^{2n\pi ix}$.
\end{thm}

We note in passing that finite gap potentials are dense in~$\Hh^w$ for $w\in\Mm$. More specifically, we call $q$ an \emph{\m{N}-gap potential}, if $\gm_n(q)=0$ for all $n>N$. But we do not insist, that the first $N$ gaps are all open.

\begin{thm} \label{t:dense}
The union of \m{N}-gap potentials is dense in $\Hh^w$ for $w\in\Mm$.
\end{thm}

We now turn to the converse problem of recovering the regularity of a potential from the asymptotic behaviour of its gap lengths. Here the situation is not as clear cut as for the forward problem.
Gasymov~\cite{Gas} observed that \emph{any} $L^2$-potential of the form
\[
  q = \sum_{n\ge1} q_n\e^{2n\pi ix} =
      \eval \smash[b]{\sum_{n\ge1} q_nz^n} \at_{z=\e^{2\pi ix}}
\]
is a {$0$-gap potential}. In the complex case, the gap sequence therefore need not contain \emph{any} information about the regularity of the potential.
But even in the real case the situation is not completely straightforward, as there are finite gap potentials, that are not entire functions, but have poles. Thus, although in this case  $\gm_n\sim\e^{-an}$ for \emph{any} $a>0$, we have $q_n\sim\e^{-\al n}$ only for \emph{some} $\al>0$. 

To obtain a true converse to Theorem~\ref{t:gaps} we need to exclude \emph{exponential} weights, that is, submultiplicative weights~$w$ with
\[
  \liminf_{n\to\infty}\, \frac{\log w(n)}{n} > 0.
\]
We call a weight \emph{strictly subexponential}, if 
\[
  \frac{\log w(n)}{n} \to 0 \qqt{as} n\to\infty
\]
in an eventually \emph{monotone} manner, while $w(n)$ itself is assumed to be nondecreasing for $n\ge0$. -- The following theorem extends results of \cite{DM-02}.

\begin{thm} \label{t:real}
Suppose $q\in\Ho$ is real, and its gap lengths satisfy
\[
  \sum_{n\ge1} w_n^2\n{\gm_n(q)}^2 < \infty.
\]
If $w$ is strictly subexponential, then $q\in\Hh^w$.
On the other hand, if $w$ is exponential, then $q$ is real analytic.
\end{thm}

This theorem does not extend to complex potentials because of Gasymov's observation. But Sansuc \& Tkachenko \cite{ST-96} noted that the situation can be remedied by taking into account additional spectral data. In particular, they considered the quantities
\[
  \dl_n = \mu_n-\tau_n,
\]
where $\mu_n$ denotes the Dirichlet eigenvalues of a potential and $\tau_n=(\lm_n^+ + \lm_n^-)/2$ the mid-points of its spectral gaps. 

More generally, one may consider a family of continuously differentiable \emph{alternate gap lengths} $\dl_n\maps\Ho\to\CC$, characterized by the properties that
\begin{asparaitem}[--]
\item
$\dl_n$ vanishes whenever $\lm_n^+=\lm_n^-$ has also geometric multiplicity~2, and
\item
there are real numbers $\xi_n$ such that its gradients satisfy
\[
  \d\dl_n = t_n + \bigo{1/n}, \qquad t_n = \cos 2n\pi(x+\xi_n),
\]
uniformly on bounded subsets of $\Ho$. That is,
$
  \nn*{\d_q\dl_n-t_n}_\o \le {C_\dl(\no{q})}/{n}
$
with $C_\dl$ depending only on $\no{q}\defeq\nn{q}_{\Ho}$.
\end{asparaitem}

For example, let $\sg_n$ denote the eigenvalues of the operator~$L$
with symmetric Sturm-Liouville boundary conditions
\[
  y\cos\al + y'\sin\al = 0 \qtext{on} \del[0,1].
\]
Dirichlet and Neumann boundary conditions correspond to the choices $\al=0$ and $\al=\pi/2$, respectively. Then $\sg_n\in[\lm_n^-,\lm_n^+]$ in the real case, and $\dl_n=\sg_n-\tau_n$ are alternate gap lengths. -- The following theorem extends results of \cite{DM-03,ST-96}.

\begin{thm} \label{t:complex}
Let $\dl_n$ be a family of alternate gap lengths on $\Ho$.
\begin{asparaenum}[\upshape(i)]
\item
If $q\in\Hh^w$ with $w\in\Mm$, then
\[
  \sum_{n\ge N} w_{n}^2\n{\dl_n(q)}^2 \le
  4\nw{T_Nq}^2+\frac{256}{N}\nw{q}^4
\]
for all $N$ sufficiently large, where $T_Nq = \sum_{\n{n}\ge N} q_n\e^{2n\pi ix}$.
\item
Conversely, suppose $q\in\Ho$ and
\[
  \sum_{n\ge1} w_n^2\pas{\n{\gm_n(q)}+\n{\dl_n(q)}}^2 < \infty.
\]
If $w$ is strictly subexponential, then $q\in\Hh^w$.
On the other hand, if $w$ is exponential, then $q$ is analytic.
\end{asparaenum}
\end{thm}

One may consider $\lm_n^-,\tau_n+\dl_n,\lm_n^+$ as the vertices of a \emph{spectral triangle} $\Dl_n$, and
\[
  \Gm_n(q) = \n{\gm_n(q)}+\n{\dl_n(q)}
\]
as a measure of its size, which takes the role of $\gm_n$ in the complex case. We then have the following consequence of Theorems \ref{t:gaps} and~\ref{t:complex}.

\begin{thm}
If $w$ is strictly subexponential, then
\[
  q\in\Hh^w \equi
  \sum_{n\ge1} w_n^2\Gm_n^2(q)<\infty,
\]
where $\Gm_n$ denotes the size of the \m{n}-th spectral triangle defined by the gap lengths $\gm_n$ and some alternate gap lengths $\dl_n$.
\end{thm}

We briefly look at the case of weights growing faster than exponentially, thus 
characterizing classes of entire functions. One can expect the gap lengths to decay faster than exponentially, too, albeit not at the same rate. We note a general result to this effect for \emph{strictly superexponential} weights, that is, weights $w$ with
\[
  \lim_{n\to\infty} \frac{\log w(n)}{n} = \infty.
\]
We only consider the gap lengths~$\gm_n$. The result for alternate gap lengths $\dl_n$ is exactly the same, only the lower bound for $n$ has to be augmented. See also \cite{DM-03b}.

\begin{thm} \label{t:super}
If $q\in\Hh^w$ with a strictly superexponential weight~$w\in\Ww$, then
\[
  \n{\gm_n(q)} \le 2n\exp(-n\psi(\nt)), \qquad
  \nt = \smash{\frac{n}{4\nw{q}}},
\]
for all $n\ge4\nw{q}$, where 
$\displaystyle
  \psi(r) = \min_{m\ge1} {\frac{\log r w(m)}{m}}.
$
\end{thm}

For instance, for $w_n=\exp({\n{n}^\sg})$ with $\sg>1$ one has
\[
  \psi(\nt) = c_\sg\log^{1-1/\sg}\nt
\]
with $c_\sg=\sg/(\sg-1)^{1-1/\sg}$. Djakov \& Mityagin~\cite{DM-03b} construct an example showing that as far as the order in $n$ is concerned, the resulting gap estimate can not be improved. 

We point out that the preceding theorem is not optimal for trigonometric polynomials. Consider for example the Mathieu potential
\[
  q = \mu\cos2\pi x, \qquad \mu>0.
\]
Using the just mentioned weight, we have $\nw{q}=c\mu/4$ with a certain constant~$c$ for \emph{all} $\sg>1$, and letting $\sg$ tend to infinity we obtain
\[
  \gm_n(q) \le 
  2n\exp\pasii{-n\log\frac{n}{c\mu}} =
  2n \pas{\frac{c\mu}{n}}^n.
\]
But Harrell \cite{Har} and Avron \& Simon \cite{AS} found the better exact asymptotics
\[
  \gm_n(q) = 
  8\pi^2 \pas{\frac{\mu}{8\pi^2}}^n \frac1{(n-1)!^{\,2}} \pas{1+O(n^{-2})},
\]
This result was later extended by Grigis~\cite{Gri} to more general real trigonometric polynomials, and to their spectral triangles by Djakov \& 
Mityagin~\cite{DM-03b}. These better estimates are obtained by directly evaluating an explicit representation of some coefficient -- see the end of section~\ref{s:gap}. This approach is different from the one taking in this paper and will not be reproduced here.

\smallskip

\textit{Acknowledgement.}\enspace A crucial ingredient of this paper --  section~\ref{s:real} -- was conceived during a visit to Zurich, and the author is very grateful to Thomas Kappeler and the Department of Mathematics at the University of Zurich for their hospitality.

\section{Outline}

The idea of the proof of Theorem~\ref{t:gaps}\ is due to Kappeler \& Mityagin~\cite{KM-01}. They employ a Lyapunov-Schmidt reduction, called \emph{Fourier block decomposition}. 

The aim is to determine those $\lm$ near $n^2\pi^2$ with $n$ sufficiently large, for which the equation $-y''+qy=\lm y$ admits a nontrivial 2-periodic solution~$f$. As $q$ can be considered small for large~$n$, one can expect its dominant modes to be $\e^{\pm n\pi ix}$. So it makes sense to separate these modes from the other ones by a Lyapunov-Schmidt reduction.

To this end we consider a Banach space $\Bb^w$ of \emph{2-periodic} functions, and write
\begin{align*}
  \Bb^w 
  &= \Pp_n \oplus \Qq_n \\
  &= \spn\set{e_k: \abs{k}=n} \oplus \spn\set{e_k: \abs{k}\ne n},
\end{align*}
where $e_k = \e^{k\pi ix}$. The pertinent projections are denoted by $P_n$ and $Q_n$, respectively. Then we write $-f''+qf=\lm f$ in the form
\[
  A_\lm f \defequal f''+\lm f = Vf, 
\]
where $V$ denotes the operator of multiplication with~$q$. With
\[
  f = u+v = P_nf+Q_nf,
\]
this equation decomposes into the two equations
\begin{align*}
  A_\lm u &= P_nV(u+v), \\
  A_\lm v &= Q_nV(u+v),
\end{align*}
strangely called the \m{P}- and \m{Q}-equation, respectively. 

We first solve the \m{Q}-equation by writing $v$ as a function of~$u$. This will reduce the \m{P}-equation to a two-dimensional equation with a $2\x2$ coefficient matrix $S_n$, which is singular precisely when $\lm$ is a periodic eigenvalue. The coefficients of $S_n$ then provide all the data to prove Theorem~\ref{t:gaps}, essentially as in \cite{KM-01}.

To go beyond Theorem~\ref{t:gaps} -- and this is the new ingredient -- we regard these coefficents as analytic functions of their potential in $\Ho$, and employ them to define, on any bounded ball in~$\Ho$, a near identity diffeomorphism $\Phi$ that introduces Fourier coefficients adapted to spectral gaps and preserves the regularity of potentials. That is, $p=\Phi(q)$ is in $\Hh^w$ if and only if $q$ is in $\Hh^w$ -- which will arise as an immediate consequence of the inverse function theorem.

Establishing the regularity of a potential $q$ then amounts to showing that $\Phi(q)$ is in $\Hh^w$. In the real case, this involves a geometric argument using the gap length asymptotics and a trick to temper the resulting \m{w}-norms. In the complex case, alternate gap lengths are needed in those cases where the coefficient matrix $S_n$ is not close to a hermitean matrix to obtain the same conclusion.

\section{Preparation}

Given a weight $w$, we introduce the Banach space
\[
  \Bb^w = \set[2]{ u=\sum_{m\in\ZZ} u_me_m: \nw{u}<\infty}
\]
of complex functions of \emph{period 2} and finite $\nw{\cd}$-norm, 
\[
  \nw{u}^2 = \sum_{m\in\ZZ} w_{m/2}^2\n{u_m}^2.
\]
We assume for simplicity, and without noticable loss of generality, that the weights are also defined on $\ZZ/2$ and have the same properties. 
Obviously, $\Bb^w$ is an extension of $\Hh^w$.
On $\Bb^w$ we consider operator norms that are defined in terms of \emph{shifted \m{w}-norms}
\[
  \nn{u}_{w;i} = \nw{ue_i}.
\]
Finally, let
\[
  U_n = \set{ \lm\in\CC: \n{\Re\lm-n^2\pi^2}\le12n}.
\]

\begin{lem} \label{l:T}
If $q\in\Hh^w$ with $w\in\Mm$, then for $n\ge1$ and $\lm\in U_n$,
\[
  T_n = \VAli Q_n
\]
is a bounded linear operator on $\Bb^w$ with norm 
\[
  \nn{T_n}_{w;i} \le \smash[t]{\dfrac2n}\nw{q}
\]
for all $i\in\ZZ$.
\end{lem}

\begin{proof}
We have $A_\lm e_m = (\lm-m^2\pi^2)e_m$ for all~$m$, and for $\n{m}\ne n$, one checks that
\[
  \min_{\lm\in U_n} \n{\lm-m^2\pi^2} \ge \n{n^2-m^2} > 0.
\]
Therefore, the restriction of $A_\lm$ to the range of $Q_n$ is boundedly invertible for all $\lm\in U_n$, and for $f=\sum_{m\in\ZZ}f_me_m$,
\[
  g =
  \Ali Q_nf = 
  \sum_{\n{m}\ne n} \frac{f_m}{\lm-m^2\pi^2}\,e_m 
\]
is well defined. For the weighted \m{L^1}-norm $\nn{g}_{w,1} = \sum_{m\in\ZZ} w_{m/2}\n{g_m}$ of~$g$ we then obtain, with the help of Hölder's inequality and the preceding two lines, 
\begin{align*}
  \nn{ge_i}_{w,1} 
  &\le
  \sum_{\n{m}\ne n}  \frac{w_{(m+i)/2}\n{f_m}}{\n{n^2-m^2}} \\ 
  &\le
  \nn{f}_{w;i}\pas[3]{\,\sum_{\n{m}\ne n}\frac{1}{\n{m^2-n^2}^2} }^{1/2}.
\end{align*}
With
\[
  \sum_{\n{m}\ne n}\frac{1}{\n{m^2-n^2}^2} \le
  \frac2{n^2} \sum_{m\ge1} \frac1{m^2} \le
  \frac4{n^2},
\]
we thus have $\nn{ge_i}_{w,1}\le2\nn{f}_{w;i}/n$.
Finally, with $q=\sum_{m\in\ZZ}u_me_m$,
\[
  (Vg)e_i = 
  \sum_{m\in\ZZ} e_{m+i} \sum_{l\in\ZZ} u_{m-l}g_l =
  \sum_{m\in\ZZ} e_{m} \sum_{l\in\ZZ} u_{m-l}g_{l-i} =
  V(ge_i)
\]
and thus $(T_nf)e_i = (Vg)e_i = V(ge_i)$.
Standard estimates for the convolution of two sequences and the submultiplicity of the weights then give
\[
  \nn{T_nf}_{w;i} =
  \nw{V(ge_i)} \le
  \nw{V}\nn{ge_i}_{w,1} \le
  \frac2n \nw{q}\nn{f}_{w;i}.
\]
This holds for any $f\in\Bb^w$ and any $i\in\ZZ$, so the claim follows.
\end{proof}

Thus, if $n\ge4\nw{q}$ and $w\in\Mm$, then $T_n$ is a \m{\frac12}-contraction on $\Bb^w$ in particular with respect to the shifted norms $\nn{\cd}_{w;\pm n}$. It is this property that we actually need in section~\ref{s:gap} to bound the \m{n}-th gap lengths from above.

\section{Reduction}

Multiplying the \m{Q}-equation from the left with $\VAli$ we obtain
\[
  Vv = T_nVu+T_nVv.  
\]
If $T_n$ is a contraction on $\Bb^w$, then this equation has a unique solution, namely
\[
  Vv = \Tin T_nVu, \qquad \Tin = (I-T_n)^{-1}.
\]
Inserted into the \m{P}-equation this gives
\[
  A_\lm u = P_nVu+P_n\Tin T_nVu = P_n\Tin Vu.
\]
So the \m{P}- and \m{Q}-equation reduce to
\[
  S_nu=0, \qquad
  S_n=A_\lm-P_n\Tin V.
\]
Any nontrivial solution $u$ gives rise to a 2-periodic solution of $A_\lm f=Vf$, and vice versa. Hence, a complex number $\lm$ near $n^2\pi^2$ is a periodic eigenvalue of $q$ if and only if the determinant of $S_n$ vanishes.

The matrix representation of any operator $I$ on the two-dimensional space $\Pp_n$  is given by $(\pair{Ie_{\pm n}}{e_{\pm n}})$, where $\pair{f}{g}=\smash{\int_0^1} f\bar g\dx$. We find that
\[
  A_\lm = \mat{\lm-\sg_n&0\\0&\lm-\sg_n},
  \qquad
  P_n\Tin V = \mat{a_n&c_{-n}\\c_{n}&a_{-n}},
\]
with $\sg_n=n^2\pi^2$ and
\[
  a_n = \pair{\Tin Ve_n}{e_n}, \qquad
  c_n = \pair{\Tin Ve_n}{e_{-n}}.
\]
Moreover, looking at the series expansion of $\Tin$ one checks that $(\Tin V)^* = (\Tin V)^-$, the complex conjugate of $\Tin V$. Therefore,
\begin{align*}
  a_n = \pair{\Tin Ve_n}{e_n}
  &= \pair*{e_n}{(\Tin V)^-e_n} \\
  &= \pair*{e_n}{(\Tin Ve_{-n})^-} 
   = \pair{\Tin Ve_{-n}}{e_{-n}} = a_{-n}.
\end{align*}
That is, the diagonal of $S_n$ is homogeneous, and we have
\[
  S_n = \mat{\lm-\sg_n-a_n&-c_{-n}\\-c_{n}&\lm-\sg_n-a_n}.
\]

Incidentally, at this point we may recover Gasymov's observation for complex potentials of the form $q=\sum_{m\ge1}q_m\e^{2m\pi ix}$. In that case, $\Tin Ve_n$ is given by a power series in $\e^{2\pi ix}$ with lowest term $e_{n+2}$, whence $a_n=c_{n}=0$ and 
\[
  S_n = \mat{\lm-\sg_n&-c_{-n}\\0&\lm-\sg_n}.
\]
It follows that $\lm_n^\pm=\sg_n$ for all $n\ge1$, which is the claim.

\section{Gap Estimates}  \label{s:gap}

\begin{lem} \label{l:ac}
If $T_n$ is a \m{\frac12}-contraction on $\Bb^w$ with respect to the shifted norms $\nn{\cd}_{w;\pm n}\,$ for all $\lm\in U_n$, then
\[
  \n{a_n-q_0}_{U_n}, w_n\n{c_n-q_n}_{U_n} \le 2\nn{T_n}_{w;-n}\nw{q}.
\]
The same applies to $c_{-n}-q_{-n}$.
\end{lem}

\begin{proof}
Consider $c_n=\pair{\Tin Ve_n}{e_{-n}}$. We note that $\Tin=I+\Tin T_n$ and thus $c_n=q_n+\pair{\Tin T_nVe_n}{e_{-n}}$. In general, from 
$\pair{f}{e_{-n}}=\pair{fe_{-n}}{e_{-2n}}$ we obtain
\[
  w_n\n{\pair{f}{e_{-n}}} \le \nn{f}_{w;-n}.
\]
The claim then follows with $f=\Tin T_nVe_n$,
\[
  \nn{\Tin T_nVe_n}_{w;-n} \le 
  2\nn{T_n}_{w;-n}\nn{Ve_n}_{w;-n}
\]
by Lemma~\ref{l:T}, and $\nn{Ve_n}_{w;-n}=\nw{Ve_0}=\nw{q}$.
The other statements are proven analogously.
\end{proof}

\begin{remark}
We have to make this somewhat roundabout argument, since composition and multiplication are not associative. That is, $(T_nVe_n)e_m$ is \emph{not equal} to $T_nVe_{n+m}$, and therefore $\pair{\Tin Ve_n}{e_{-n}}$ is \emph{not equal} to $\pair{\Tin q}{e_{-2n}}$. The estimate of the latter would be much more straightforward and would not require shifted norms.
\end{remark}

\begin{lem} \label{l:gap}
If Lemma~\ref{l:ac} applies and $n\ge2\nw{q}$, then the determi\-nant of $S_n$ has exactly two complex roots $\xi_-$, $\xi_+$ in $U_n$, which are contained in 
\[
  D_n = \set{ \lm: \n{\lm-\sg_n}\le6\nw{q}}
\] 
and satisfy
\[
  \n{\xi_+-\xi_-}^2 \le 9\n{c_nc_{-n}}_{U_n}.
\]
\end{lem}

A more precise location of these roots is obtained in the proof of Lemma~\ref{l:skew} below. But for now, this simpler statement suffices.

\begin{proof}
Write $\det S_n = g_+g_-$ with
\[
  g_\pm = \lm-\sg_n-a_n\mp \phi_n, \qquad
  \phi_n = \sqrt{c_nc_{-n}},
\]
where the choice of the branch of the root is immaterial. In view of the preceding lemma and the normalization $w_n\ge1$,
\[
  \n{a_n}_{U_n} + \n{\phi_n}_{U_n} \le 4\nw{q} <
  \eval \n*{\lm-\sg_n}\at_{U_n\setdif D_n}. 
\]
It follows with topological degree theory that both $g_+$ and $g_-$ have exactly one root in $D_n$, while they obviously have no roots in $U_n\setdif D_n$.

To estimate the distance of these roots, let $\xi_+$ be the root of $g_+$, and
\[
  K_n = \set{\lm: \n{\lm-\xi_+}\le 3r_n}, \qquad
  r_n \defequal \n{\phi_n}_{U_n} \le 2\nw{q} \le n.
\]
The function $h=\lm-\sg_n-a_n(\xi_+)-\phi_n(\xi_+)$ vanishes at~$\xi_+$, thus
$
  \eval\n{h}\at_{\del K_n} = 3r_n
$.
On the other hand,
\[
  \n{h-g_-}_{\del K_n} \le
  \n{a_n-a_n(\xi_+)}_{K_n} + 2\n{\phi_n}_{U_n} <
  r_n + 2r_n = 
  \eval\n{h}\at_{\del K_n},
\]
since $\n{\del_\lm a_n}_{K_n}\le\n{a_n}_{U_n}/4n\le1/4$ by Cauchy's inequality.
It follows again with topological degree theory that $g_-$ has on $K_n$ the same index with respect to $0$ as $h$, namely~1. Hence, the second root $\xi_-$ of $\det S_n$ is located in~$K_n$, which gives the claim.
\end{proof}

We now prove Theorem~\ref{t:gaps}. If $q\in\Hh^w$ with $w\in\Mm$ and $n\ge4\nw{q}$, then $T_n$ is a \m{\frac12}-contraction on $\Hh^w$ by Lemma~\ref{l:T} with respect to all shifted norms. So Lemma~\ref{l:gap} applies, giving us two roots of $\det S_n$ in $D_n\subset U_n$. Now the union of all strips $U_n$ covers the right complex half plane. Since $\lm_n^\pm\sim n^2\pi^2$ asymptotically, and since there are no periodic eigenvalues in $\bigcup_{n\ge 4\nw{q}} (U_n\setdif D_n)$, those two roots in $D_n$ must be the periodic eigenvalues $\lm_n^\pm$. Thus,
\[
  \n{\gm_n(q)}^2 =
  \n{\xi_+-\xi_-}^2 \le 
  9\n{c_nc_{-n}}_{U_n} \le
  \frac92\n{c_n}_{U_n}^2 + \frac92\n{c_{-n}}_{U_n}^2.
\]
By Lemmas~\ref{l:T} and~\ref{l:ac},
\[
  w_n\n{c_n} \le w_n\n{q_n} + 2\nn{T_n}_{w;-n}\nw{q} \le 
  w_n\n{q_n} + \frac4n\nw{q}^2.
\]
Both estimates together then lead to
\[
  \frac{1}{9}w_n^2\n{\gm_n(q)}^2 \le
  w_n^2\n{q_{n}}^2+w_n^2\n{q_{-n}}^2 + \frac{32}{n^2}\nw{q}^4.
\]
Summing up, we arrive at
\begin{align*}
  \frac{1}{9}{\sum_{n\ge N} w_n^2\n{\gm_n(q)}^2}
  &\le {\sum_{\n{n}\ge N} w_n^2\n{q_n}^2} +
       \nw{q}^4 \sum_{n\ge N} \frac{32}{n^2} \\
  &\le \nw{T_Nq}^2 + \frac{64}N\nw{q}^4.
\end{align*}
This gives Theorem~\ref{t:gaps}.

Incidentally, if we just make use of $w_n\n{c_n}\le2\nw{q}$, then we get the individual gap estimate
\[
  w_n\n{\gm_n(q)} \le 6\nw{q}.
\]
We will use this observation in section~\ref{s:super}.
Finally, we note that Lemma~\ref{l:gap} together with the expansion
\[
  c_n = \pair{\Tin Ve_n}{e_{-n}} = \sum_{\nu\ge0} \pair*{T_n^\nu Ve_n}{e_{-n}}
\]
allows for an effective control of $\gm_n$ for trigonometric polynomials~$q$, since then some first terms in the series vanish -- see \cite{AS,DM-03b}.

\section{Adapted Fourier Coefficients}

The $2\x2$-matrix $S_n$ contains all the information we need about the \m{n}-th periodic eigenvalues of a potential, at least in the real case. Even more to the point, the diagonal of $S_n$ vanishes at a unique point
\[
  \lm = \al_n(q),
\]
and it suffices to consider its off-diagonal elements at $\lm=\al_n(q)$. We will make use of these values to define a real analytic \emph{adapted Fourier coefficient map}, which allows us to prove the regularity results by invoking the inverse function theorem.

We begin by observing that the coefficients $a_n$ and $c_n$ do not depend on the underlying space~$\Hh^w$, but are rather defined on appropriate balls in $\Ho$, with estimates depending on the regularity of~$q$. To make this precise, we introduce the notation
\[
  B^w_m = \set{q\in\Hh^w:4\nw{q}\le m}
\]
and note that 
$
  B^w_m\subset \Bo_m \defequal \set{q\in\Ho:4\no{q}\le m}\subset\Ho
$
for all $w\in\Mm$.

We assume from now on that $q$ has mean value zero,
\[
  \mean*{q} = \int_0^1 q(x)\dx = q_0 = 0,
\]
since adding a constant to the potential $q$ shifts its entire spectrum by this amount, but does not affect the lengths of its gaps.

\begin{lem} \label{l:be}
For $n\ge m$, the coefficients $a_n$ and $c_n$ are analytic functions on $U_n\x\Bo_m$ with
\[
  \n{a_n}_{U_n\x\Bo_m}, w_n\n{c_n-q_n}_{U_n\x B^w_m} \le
  \frac{m^2}{4n}
\]
for all weights $w\in\Mm$. The same applies to $c_{-n}-q_{-n}$.
\end{lem}

\begin{proof}
The estimates follow from Lemmas \ref{l:T} and~\ref{l:ac}, the normalization $q_0=0$ and $\nw{q}\le m/4$ on $B^w_m$. The analytic dependence on $q$ then follows from the series expansion of $\Tin$.
\end{proof}

\begin{lem}
For $m\ge1$ and each $n\ge m$, there exists a unique real analytic function
\[
  \al_n\maps \Bo_m\to\CC, \qquad \n{\al_n-\sg_n}_{\Bo_m} \le
  \frac{m^2}{4n},
\]
such that $\al_n=\sg_n+a_n(\al_n,\cd)$ identically on~$\Bo_m$.
\end{lem}

\begin{proof}
Consider the fixed point problem for the operator $T$,
\[
  T\al \defequal \sg_n + a_n(\al,\cd),
\]
on the ball of all real analytic functions $\al\maps \Bo_m\to\CC$ with $\n*{\al-\sg_n}_{\Bo_m} \le m^2/4n$. Since clearly $m^2/4n\le n$ by assumption, each such function $\al$ maps $\Bo_m$ into the disc $D_n=\set*{\n*{\lm-\sg_n}\le n}\subset U_n$, and so
\[
  \n{T\al-\sg_n}_{\Bo_m} \le \n{a_n}_{U_n} \le m^2/4n
\]
in view of Lemma~\ref{l:be}. Moreover, $T$ contracts by a factor
\[
  \n{\del_\lm a_n}_{D_n} \le 
  \frac{\n{a_n}_{U_n}}{2n} \le
  \frac{m^2}{8n^2} \le \frac18,
\]
using Cauchy's estimate. Hence, we find a unique fixed point $\al_n=T\al_n$ with the properties as claimed.
\end{proof}

In the following we let $\al_{-n}=\al_n$ to simplify notation.
--
For each $m\ge1$ we now define a map $\Phi_m$ on $\Bo_m$ by
\[
  \Phi_m(q) =
  \sum_{\n{n}<M_m} q_ne_{2n} +
  \sum_{\n{n}\ge M_m} c_n(\al_n(q),q)e_{2n},
\]
where $M_m = 2^{10}m^2$. Thus,  for $\n{n}\ge M_m$ the Fourier coefficients of the 1-periodic function $p=\Phi_m(q)$ are
$
  p_n = c_n(\al_n) 
$,
and
\[
  S_n(\al_n) = S_n(\al_n,p) = \mat{0&-p_{-n}\\-p_{n}&0}.
\]
These new Fourier coefficients are adapted to the lengths of the corresponding spectral gaps, whence we call $\Phi_m$ the \emph{adapted Fourier coefficient map} on $\Bo_m$.

\begin{prp} \label{Phi}
For each $m\ge1$, $\Phi_m$ maps $\Bo_m$ into $\Ho$. Its restrictions to $B_m^w$ are real analytic diffeomorphisms
\[
  \eval\Phi_m\at{B^w_m}\maps B^w_m\to\Phi_m(B^w_m)\subset\Hh^w
\]
for every weight $w\in\Mm$, such that
\[
  \frac12\nw{q}\le\nw{\Phi_m(q)}\le2\nw{q}
\]
for $q\in B^w_m$ and $\nn{D\Phi_m-I}_{B^w_m}\le1/8$.
\end{prp}

\begin{proof}
\kern-1pt
Since $\al_n$ maps $\Bo_{2m}$ into $U_n$ for $n\ge2m$, each coefficient $c_n(\al_n(q),q)$
is well defined for $q\in \Bo_{2m}$, and
\[
  w_n\n{c_n(\al_n)-q_n}_{B^w_{2m}} \le 
  w_n\n{c_n-q_n}_{U_n\x B^w_{2m}} \le \smash{\frac{m^2}{n}}
\]
by Lemma~\ref{l:be}. Hence the map $\Phi_m$ is defined on $\Bo_{2m}$, and
\begin{align*}
  \nn{\Phi_m-\id}_{w,B^w_{2m}}^2
   &= \sum_{\n{n}\ge M_m} w_n^2\n{c_n(\al_n)-q_n}^2_{B^w_{2m}} \\
   &\le \sum_{n\ge M_m} \frac{2m^4}{n^2} 
    \le \frac{4m^4}{M_m}
    =   \frac{m^2}{256}
\end{align*}
by our choice of $M_m$. Therefore, $\Phi_m\maps B^w_{2m}\to\Hh^w$ with 
$\nn{\Phi_m-\id}_{w,B^w_{2m}}\le \smash{\dfrac{m}{16}}$. Cauchy's estimate then yields
\[
  \nn{D\Phi_m-I}_{w,B^w_m} \le
  \frac{2}{m}\nn{\Phi_m-\id}_{w,B^w_{2m}} \le 
  \frac1{8}
\]
Now the result follows by standard arguments and the fact that $\Phi_m(0)=0$.
\end{proof}

We now proof Theorem~\ref{t:dense}. Fix any ball $B_m=B^w_m$. The \m{n}-th gap of $q\in B_m$, with $n\ge M_m$, is collapsed if the \m{n}-th and \m{-n}-th Fourier coefficients of $\Phi_m(q)$ vanish, since then $S_n$ vanishes identically at $\lm=\al_n(q)$. Consequently, if
\[
  \Phi_m(q) \in \Gg_N = \spn\set{e_{2k}: \n{k}\le N},
\]
$N$ sufficienly large, then $q\in B_m$ is an \m{N}-gap potential.
The union of the spaces $\Gg_N$ is dense in $\Hh^w$. Since $\Phi_m$ is a diffeomorphism on $B_m$, the family of \m{N}-gap potentials in $B_m$ is also dense.
Since $B_m$ was arbitrary, this proves the theorem.

\section{Regularity: The Abstract Case}

From an abstract point of view, establishing the regularity of a potential $q$ amounts to the following observation about its adapted Fourier coefficients.

\begin{prp} \label{reg}
If $q\in\Bo_m\,$ for some $m\ge1$, and 
\[
  \Phi_m\maps \Bo_m\owns q \mapsto p=\Phi_m(q) \in B_{m/2}^w
\]
for some weight~$w\in\Mm$, then $q\in B_m^w\subset\Hh^w$.
\end{prp}

\begin{proof}
The map $\Phi_m$ is defined on $\Bo_m$ and a real analytic diffeomorphism onto its image
\[
  \Bt_m^\o = \Phi_m(\Bo_m) \subset \Ho.
\]
At the same time, for any weight $w\in\Mm$, $\Phi_m$ is also defined on $B_m^w\subset \Bo_m\cap\Hh^w$ and a real analytic diffeomorphism onto its image
\[
  \Bt_m^w = \Phi_m(B_m^w) \subset \Hh^w.
\]
Moreover, this image contains $B_{m/2}^w$ by Proposition~\ref{Phi}. Thus, if  $\Phi_m$ maps $q\in\Bo_m$ to
\[
  p = \Phi_m(q) \in B_{m/2}^w,
\]
then we must have
\[
  q = \eval\Phi_m^{-1}\at B_{m/2}^w (p) \in B_m^w \subset \Hh^w,
\]
thus establishing the regularity of~$q$.
\end{proof}

\section{Regularity: The Real Case} \label{s:real}

\begin{prp} \label{sub}
Suppose $q\in\Bo_m\,$ for some $m\ge1$, and 
\[
  \Phi_m(q)\in\Hh^w.
\]
If $w$ is strictly subexponential, then also $q\in\Hh^w$.
If, however, $w$ is expo\-nen\-tial, then $q\in\Hh^{v_\ep}$ for all sufficiently small positive~$\ep$, where $v_\ep = \e^{\ep\n\cd}$.
\end{prp}

Note that in contrast to Proposition~\ref{reg} we do \emph{not} assume that $\Phi_m(q)\in B_{m/2}^w$. That is, we have no \emph{a priori} bound on $\nw{\Phi_m(q)}$. To reduce the situation to the former setting nonetheless, we introduce a modified weight $w_\ep$, which tempers a potentially large chunk of $\nw{\Phi_m(q)}$ arising from finitely many modes, without affecting the asymptotic behaviour of~$w$ in the case of subexponential weights. The crucial ingredient is the following lemma.

\begin{lem} \label{l:w}
If $w$ is either strictly subexponential or exponential, then 
\[
  w_\ep \defequal \min(v_\ep,w) \in \Mm
\]
for all sufficiently small positive $\ep$.
\end{lem}

\begin{proof}
If $w$ is exponential, then $w_\ep=v_\ep$ for all sufficiently small positive~$\ep$, and there is nothing to do.

So assume $w$ is strictly subexponential.
All the required properties are readily verified for $w_\ep$, except  submultiplicity. To do this, let
\[
  \wt = \log w, \qquad \wt_\ep = \log w_\ep.
\]
As $\wt(n)/n$ converges eventually monotonically to zero by assumption, there exists for each sufficiently small $\ep>0$ an integer $N_\ep$ such that
\[
  \frac{\wt(i)}{i} \ge \ep > \frac{\wt(n)}{n} > \frac{\wt(m)}{m} 
  \qtext{for}
  1\le i\le N_\ep < n < m.
\]
It follows that
\[
  \wt_\ep = \begin{cases} 
               \vt_\ep &\ent{on} [0,N_\ep], \\
               \wt     &\ent{on} (N_\ep,\infty).
            \end{cases}
\]
To check for the subadditivity of $\wt_\ep$ for $0\le n\le m$, we consider the four possible cases
\[ \begin{aligned}
  \text{(a)}& & n+m &\le N_\ep,   &\qquad \text{(c)}& & n\le N_\ep &< m,  \\
  \text{(b)}& & m &\le N_\ep < n+m,  &    \text{(d)}& &      N_\ep &<n.\\
\end{aligned} \]
Case (a) reduces to $\tilde v_\ep$, and case (d) reduces to $\wt$. In case (b),
\[
  \wt_\ep(n+m) \le \vt_\ep(n+m) = \vt_\ep(n)+\vt_\ep(m) = \wt_\ep(n)+\wt_\ep(m).
\]
Finally, in case (c), using the monotonicity property in the second line,
\begin{align*}
  \wt_\ep(n+m)
  &=   \frac{\wt(n+m)}{n+m}n+\frac{\wt(n+m)}{n+m}m\\
  &\le \ep n + \wt(m) \\
  &=   \wt_\ep(n)+\wt_\ep(m).
\end{align*}
This establishes the subadditivity of $\wt_\ep$ for nonnegative arguments. The remaining cases all reduce to the monotonicity of $\wt_\ep$, that is,
\[
  \wt_\ep(n-m) \le \wt_\ep(n+m) \le \wt_\ep(n)+\wt_\ep(m)
\]
for $0\le m\le n$.
\end{proof}

\begin{proof}[Proof of Proposition~\ref{sub}]
We may assume that $m \ge 32\no{q}$, since the assumptions are not affected by increasing~$m$. For $p = \Phi_m(q)$ we have
\[
  \no{p} \le 2\no{q}
\]
by Proposition~\ref{Phi}. On the other hand, $p\in\Hh^w$ by assumption, so
\[
  \nw{p} < \infty.
\]
Given that $p\ne0$ without loss of generality, we can therefore choose $N$ so large that
\[
  \nw{T_Np} \le \no{p},
\]
where $T_Np=\sum_{\n{n}\ge N}p_ne_{2n}$. With respect to the weight $w_\ep = \min(v_\ep,w)$ with $\ep\le1/2N$ sufficiently 
small, we then have
\begin{align*}
  \nn{p}_{w_\ep}^2
  &= \nn{p-T_Np}_{w_\ep}^2 + \nn{T_Np}_{w_\ep}^2 \\
  &\le \nn{p-T_Np}_{v_\ep}^2 + \nw{T_Np}^2 \\
  &\le \e^{2N\ep}\no{p}^2 + \no{p}^2 \\
  &\le 4\no{p}^2,
\end{align*}
or
\[
  4\nn{p}_{w_\ep} \le 8\no{p} \le 16\no{q} \le \frac{m}{2}.
\]
Thus, $p\in B_{m/2}^{w_\ep}$, whence
\[
  q = \Phi_m^{-1}(p) \in B_m^{w_\ep} \subset \Hh^{w_\ep}
\]
by Proposition~\ref{reg}. The claim follows by noting that $\Hh^{w_\ep}=\Hh^w$ for strictly subexponential weights, and $\Hh^{w_\ep}\subset\Hh^{v_\ep}$ for exponential weights and all small $\ep>0$.
\end{proof}

To obtain Theorem~\ref{t:real} from Proposition~\ref{sub}, we now want to bound the Fourier coefficients of $p=\Phi_m(q)$ in terms of the gap lengths of~$q$. For real~$q$, this is fairly straightforward, since then
\[
  S_n = \mat{\lm-\sg_n-a_n&-c_{-n}\\-c_{n}&\lm-\sg_n-a_n}
\]
is hermitean, and $\det S_n$ is a real function of $\lm$, which is close to the standard parabola with minimum near $\al_n$ and minimal value about $-p_np_{-n}=-\n{p_n}^2$. The distance of its two roots is then about $\n{p_n}$. With foresight to the complex case, however, we want to consider a more general situation.

\begin{lem}  \label{l:skew}
Let $q\in \Bo_m\,$ for some $m\ge1$ and $p=\Phi_m(q)$. If
\[
  \frac14 \le \smash[b]{\n{\frac{p_n}{p_{-n}}}} \le 4
\]
for any $n\ge M_m$, then
\[
  \n{p_np_{-n}} \le \n{\gm_n(q)}^2 \le 9\n{p_np_{-n}}.
\]
\end{lem}

\begin{proof}
As in the proof of Lemma~\ref{l:gap}, write $\det S_n = g_+g_-$ with
\[
  g_\pm = \lm-\sg_n-a_n\mp \phi_n, \qquad
  \phi_n = \sqrt{c_nc_{-n}}.
\]
The assumptions imply that
\[
  \xi_n \defequal \phi_n(\al_n) = \sqrt{p_np_{-n}} \ne 0,
  \qquad
  r_n \defequal \n{\xi_n} > 0,
\]
so we may choose a fixed sign of the root locally around~$\al_n$.

We compare $g_+$ with
$
  h_+ = \lm-\sg_n-a_n(\al_n)-\phi_n(\al_n)
$
on the disc 
\[
  D_n^+ = \set{\lm:\n{\lm-(\al_n+\xi_n)}\le r_n/2}.
\] 
As $h_+(\al_n+\xi_n) = \xi_n - \phi_n(\al_n) = 0$,
we have 
\[
  \eval\n{h_+}\at_{\del D_n^+} = \frac{r_n}{2}.
\] 
On the other hand, we momentarily show that on $\Do_n = \set{\lm:\n{\lm-\al_n}\le2r_n}$,
\[
  \n{\del_\lm a_n}_{\Do_n} \le \frac1{18}, 
  \qquad
  \n{\del_\lm\phi_n}_{\Do_n} \le \frac1{6},
\]
which will give
\begin{align*}
  \n{h_+-g_+}_{D_n^+}
  &\le \n{a_n-a_n(\al_n)}_{\Do_n} + \n{\phi_n-\phi_n(\al_n)}_{\Do_n} \\
  &\le \frac{r_n}{9} + \frac{r_n}{3}  \\
  &<   \frac{r_n}{2} 
   =   \eval\n{h_+}\at_{\del D_n^+}.
\end{align*}
It follows that the unique root of $g_+$ within $D_n$ must be contained in $D_n^+$, that is,
\[
  \xi_+ = \lm_n^+ \in D_n^+.
\]
Similarly,
$
  \xi_-=\lm_n^-\in D_n^- = \set{\lm:\n{\lm-(\al_n-\xi_n)}\le r_n/2}
$.
Since $\n{\xi_n}=r_n$, we conclude that
\[
  r_n \le \n{\gm_n} = \n{\lm_n^+-\lm_n^-} \le 3r_n, 
\]
which is the claim.

\begin{figure}[tbp]
\begin{center}
\includegraphics[scale=1.0]{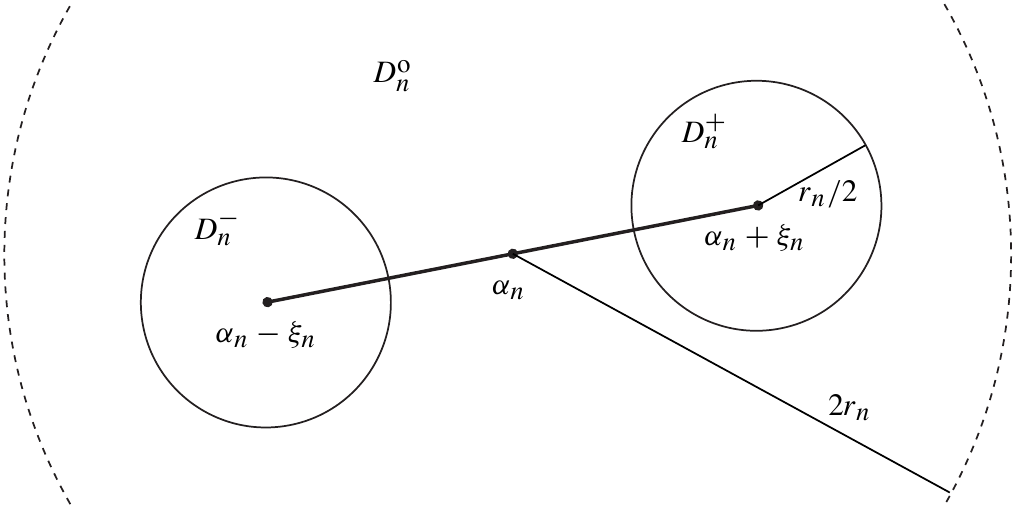}
\end{center}
\end{figure}

It remains to prove the estimates for $\del_\lm  a_n$ and $\del_\lm\phi_n$. In view of Lemma~\ref{l:be} and Cauchy's inequality,
\[
  \n{\del_\lm a_n}_{\Do_n}, \n{\del_\lm c_n}_{\Do_n} \le \frac1{36},
\]
since the distance of $\Do_n$ to the boundary of $U_n$ is at least $9n$. With $c_n(\al_n)=p_n$, $r_n=\sqrt{\n{p_np_{-n}}}$ and the hypotheses of the lemma, we get
\[
  \n{c_n-p_n}_{\Do_n} \le \frac{r_n}{16}, 
\] or \[
  \frac{r_n}{2} \le \n{p_n} \le 2r_n.
\]
Hence,
\[
  \smash[t]{
  \frac{7}{16}\,r_n = \pas{\frac12-\frac1{16}}\,r_n \le
  \eval\n{c_n}\at_{\Do_n} \le
  \pas{2+\frac{1}{16}}\,r_n = \frac{33}{16}\,r_n,
  }
\]
and therefore
\[
  \n{\frac{c_n}{c_{-n}}}_{\Do_n}, \n{\frac{c_{-n}}{c_{n}}}_{\Do_n} \le 6.
\]
Differentiating $\phi_n=\sqrt{c_nc_{-n}}$ with respect to $\lm$ we finally obtain
\[
  \n{\del_\lm\phi_n}_{\Do_n} \le
  3\pasi{\n{\del_\lm c_n}_{\Do_n}+\n{\del_\lm c_{-n}}_{\Do_n}} \le \frac16
\]
as claimed. This completes the proof.
\end{proof}

We now prove Theorem~\ref{t:real}. Suppose $q\in\Ho$ is real, and its gap lengths satisfy
\[
  \sum_{n\ge1} w_n^2\n{\gm_n(q)}^2 < \infty.
\]
Fix $m\ge4\no{q}$, and consider the coefficients $p_n=c_n(\al_n)$ for $\n{n}\ge M_m$. As $q$ is real, $p_{-n}=\bar p_n$. So the preceding lemma applies, giving
\[
  \n{p_{-n}} = \n{p_n} \le \n{\gm_n(q)}, \qquad  n\ge M_m.
\]
But this means that $p=\Phi_m(q)\in\Hh^w$, and the result follows with Proposition~\ref{sub}.

\section{Regularity: The Complex Case}

Let $\dl_n$ be a family of alternate gap lengths. Since the involved constant $C_\dl$ is supposed to depend only on $\no{q}$, we have on any ball $\Bo_m$ an estimate
\[
  \no{\d_q\dl_n-t_n} \le \frac{C_m}{n},
  \qquad
  t_n = \cos2n\pi(x+\xi_n).
\]
with a constant $C_m$ depending only on~$m$.

\begin{lem}
If $q\in\Bo_m$ for some $m\ge1$ and $p=\Phi_m(q)$, then
\[
  \n{\dl_n(q)-\pas{\kp p_n+\bar\kp p_{-n}}} \le
  \frac14(\n{p_n}+\n{p_{-n}})
\]
for $n\ge N_m\defeq\max(M_m,16\,C_m)$, where $\kp=\e^{2n\pi i\xi_n}/2$.
\end{lem}

\begin{proof}
Given $p=\sum_{k\ne0} p_{k}e_{2k} = \Phi_m(q)$ and $n\ge M_m$, let
\[
  p^\o = \sum_{0<\n{k}\ne n} p_{k}e_{2k},
  \qquad
  q^\o = \Phi_m\inv(p^\o).
\]
Then the \m{\n{n}}-th Fourier coefficients of $\Phi_m(q^\o)$ vanish, which means that $\al_n(q^\o)$ is a double periodic eigenvalue of $q^\o$ of geometric multiplicity~2. Therefore,
\[
  \dl_n(q^\o) = 0.
\]
With $q^t=tq+(1-t)q^\o$, we get
\begin{align*}
  \dl_n(q)
  &= \dl_n(q)-\dl_n(q^\o) \\
  &= \int_0^1 \pair*{\d\dl_n(q^t)}{q-q^\o}\dt \\
  &= \pair*{t_n}{q-q^\o} + \pair*{\tht_n}{q-q^\o}
\end{align*}
with $\tht_n=\int_0^1 (\d\dl_n-t_n)(q^t)\dt$. Moreover,
\[
  q-q^\o = p-p^\o + \Th_m(p-p^\o),
\]
with $\Th_m = \int_0^1 (D\Phi_m\inv-I)(\Phi_m(q^t))\dt$. Altogether we obtain
\begin{align*}
  \dl_n(q) 
  = \pair*{t_n}{p-p^\o} 
     + \pair*{t_n}{\Th_m(p-p^\o)} + \pair*{\tht_n}{q-q^\o}.
\end{align*}
The identity $\pair*{t_n}{p-p^\o}=\kp p_n+\bar\kp p_{-n}$, the estimates
\begin{align*}
  \no{\tht_n} \le \frac{C_m}{n} \le \frac1{16},
  \qquad
  \nn{\Th_m}_{L(\Ho,\Ho)} \le \frac1{6} 
\end{align*}
by Lemma~\ref{Phi}, as well as $\no{t_n}\le1$ and 
$\no{p-p^\o}\le\n{p_n}+\n{p_{-n}}$ then give the claim.
\end{proof}

We now prove Theorem~\ref{t:complex}. Given $q\in B^w_m$ and assuming $n\ge N_m$, we have by the preceding lemma
\[
  \n{\dl_n(q)}^2 \le 
  \pas{\n{p_n}+\n{p_{-n}}}^2 \le
  2\n{c_n}_{U_n}^2 + 2\n{c_{-n}}_{U_n}^2.
\]
We are thus in exactly the same situation as at the end of section~\ref{s:gap}, modulo a factor $4/9$. So we get
\[
  \sum_{n\ge N} w_n^2\n{\dl_n(q)}^2 
  \le 4\nw{T_Nq}^2 + \frac{256}{N}\nw{q}^4
\]
for all $N\ge N_m$, as well as
\[
  w_n\n{\dl_n(q)} \le 4\nw{q}
\]
for all $n\ge N_m$.
This establishes (i) of Theorem~\ref{t:complex}.

To prove the converse statement (ii), we only need to augment the proof of Theorem~\ref{t:real} in the case where $p_n$ and $p_{-n}$ are not about the same size. So suppose $q$ is in $\Ho$ with
\[
  \sum_{n\ge1} w_n^2\pas{\n{\gm_n(q)}+\n{\dl_n(q)}}^2 < \infty.
\]
Fix $m\ge4\no{q}$, and consider the coefficients $p_n=c_n(\al_n)$ for $\n{n}\ge M_m$. For any such~$n$, for which the hypotheses of Lemma~\ref{l:skew} are satisfied, we have
\[
  \n{p_n}, \n{p_{-n}} \le 2\n{\gm_n(q)}.
\]
Otherwise, we may assume that $\n{p_n}\ge4\n{p_{-n}}$, and we can use the preceding lemma to the effect that
\begin{align*}
  \n{\dl_n(q)}
  &\ge \n{\kp p_n+\bar\kp p_{-n}} - \frac14\pas{\n{p_n}+\n{p_{-n}}} \\
  &\ge \frac12\cd\frac34 \n{p_n} - \frac14\cd\frac54 \n{p_n} \\
  &= \frac1{16} \n{p_n}.
\end{align*}
So in this case we get
\[
  \n{p_n}, \n{p_{-n}} \le 16 \n{\dl_n(q)}.
\]
We again conclude that $p=\Phi_m(q)\in\Hh^w$, and the result follows with Proposition~\ref{sub}. This proves Theorem~\ref{t:complex}.

\section{Superexponential Weights}  \label{s:super}

We prove Theorem~\ref{t:super} by using the gap estimates already established for exponential weights. If $q\in\Hh^w$ with a strictly superexponential weight~$w$, then in particular $q\in\Hh^{a}$ for all $a\ge0$, where $a$ stands for the exponential weight $\exp(a\n\cd)$. Given any $n\ge4\nw{q}$, we may thus choose
\[
  a = \psi(\nt) = \min_{m\ge1}\frac{\log\nt w(m)}{m} \ge 0, \qquad
  \nt = \frac{n}{4\nw{q}} \ge 1.
\]
Then $\e^{am}/w_m\le\nt$ for all $m\ge1$, and consequently
\[
  \nn{q}_{a} \le \sup_{m\ge1} \frac{\e^{am}}{w_m} \nw{q} \le
  \nt\nw{q} = \frac{n}{4}.
\]
We may thus apply the individual gap estimate given at the end of section~\ref{s:gap} to obtain
\[
  \n{\gm_n(q)} \le \frac{6}{a_n}\nn{q}_{a} \le
  2n\,\e^{-an} = 
  2n\,\e^{-n\psi(\nt)}.
\]
This is the claim, and Theorem~\ref{t:super} is proven.

Incidentally, the result is the same for alternate gap lengths, using the individual estimate given at the end of the preceding section. We only have to assume in addition that $n\ge N_m$, a constant depending only on $\no{q}$.

\section{Extensions}

\paragraph{Subexponential weights}
Our definition of a strictly subexponential weight is chosen to allow for a convenient  hypothesis of Lemma~\ref{l:w}. But we might as well \emph{define} a weight $w\in\Mm$ to be strictly subexponential, if $\log w(n)/n\to0$ and
\[
  \min(w,v_\ep) \in \Mm
\]
for all sufficiently small positive~$\ep$. Then Theorems~\ref{t:real} and~\ref{t:complex} remain valid.

\paragraph{$L^p$-spaces}
For the sake of brevity and clarity we restricted ourselves to spaces $\Hh^w$ defined in terms of \m{L^2}-type norms. But we may also consider the spaces
\[
  \Hh^w_r = \set[2]{q=\sum_{n\in\ZZ}q_ne_{2n}: \nn{q}_{w,r}<\infty}
\]
for $1\le r\le\infty$, where
\begin{gather*}
  \nn{q}_{w,r}^r = \sum_{n\in\ZZ} w_n^r\n{q_n}^r, \qquad 1\le r<\infty, \\
  \nn{q}_{w,\infty} = \sup_{n\in\ZZ} w_n\n{q_n}.
\end{gather*}
The shifted norms $\nn{\cd}_{w,r;i}$ are defind analogously.
The results remain the same, except for some minor quantitative aspects of constants and thresholds. The only new ingredient is an extended version of Lemma~\ref{l:T}.

\begin{proclaim}[Lemma \ref{l:T}-R]
If $q\in\Hh^w_r$ with $w\in\Mm$, then for $n\ge1$ and $\lm\in U_n$,
\[
  T_n=\VAli Q_n
\]
is a bounded linear operator on $\Bb^w_r$ with norm $\nn{T_n}_{w,r;i}\le \smash{\dfrac{c_r}{n}}\nn{q}_{w,r}$ for all $i\in\ZZ$, where $c_1=1$ and otherwise
\[
  c_r^s = \sum_{m\ge1} \frac{2}{m^s}, \qquad 
  s=\frac{r}{r-1}.
\]
\end{proclaim}

\begin{proof}
Consider the case $1<r<\infty$. As in the proof of Lemma~\ref{l:T}, we may write
\[
  g =
  \Ali Q_nf = 
  \sum_{\n{m}\ne n} \frac{f_m}{\lm-m^2\pi^2}\,e_m =
  \sum_{m\in\ZZ} g_me_m,
\]
and by Hölder's inequality for $r^{-1}+s^{-1}=1$ we get
\[
  \nn{ge_i}_{w,1} \le
  \nn{f}_{w,r;i}
  \pas[3]{\,\smash[b]{\sum_{\n{m}\ne n}}\frac{1}{\n{m^2-n^2}^s} }^{1/s}.
\]
One verifies that
\[
  \sum_{\n{m}\ne n}\frac{1}{\n{m^2-n^2}^s} \le
  \frac1{n^s} \sum_{m\ge1} \frac2{m^s} \le
  c_r^s,
\]
so that $\nn{ge_i}_{w,1} \le c_r\nn{f}_{w,r;i}$. By standard estimates for the convolution of two sequences and the submultiplicity of the weights, one then arrives at
\[
  \nn{T_nf}_{w,r;i} =
  \nn{Vg}_{w,r;i} \le
  \nn{V}_{w,r}\nn{ge_i}_{w,1} \le
  c_r\nn{q}_{w,r}\nn{f}_{w,r;i}.
\]
This holds for any $f\in\Hh^w_r$, so the claim follows for $1<r<\infty$. The remaining cases are handled analogously.
\end{proof}



\begin{thebibliography}{\hskip12.5pt}

\small
\itemsep1pt plus .2pt
\raggedright
\catcode`\…\active \def…{.\thinspace\ignorespaces}

\bibitem{AS}
\sc J…Avron \& B…Simon,
\rm The asymptotics of the gaps in the Mathieu equation.
\it Ann. of Physics \bf 134 \rm(1981), 76--84.

\bibitem{DM-02}
\sc P…Djakov \& B…Mityagin,
\rm Smoothness of Schrödinger operator potential in the case of Gevrey type asymptotics of the gaps. 
\it J. Funct. Anal. \bf 195 \rm(2002), 89--128.

\bibitem{DM-03}
\sc P…Djakov \& B…Mityagin,
\rm Spectral triangles of Schrödinger operators with complex potentials.  
\it Selecta Math. (N.S.) \bf 9 \rm(2003), 495--528.

\bibitem{DM-03b}
\sc P…Djakov \& B…Mityagin,
\rm Spectral gaps of the periodic Schrödinger operator when its potential is an entire functin.  
\it Adv. in Appl. Math. \bf 31 \rm(2003), 562--596.
  
\bibitem{Gas}
\sc M…G…Gasymov,
\rm Spectral analysis of a class of second order nonselfadjoint differential operators.
\it Funct. Anal. Appl. \bf 14 \rm (1980), 14--19.

\bibitem{GKP}
\sc B…Grébert, T. Kappeler \& J…Pöschel,
\rm A note on gaps of Hill's equation.
\it Int. Math. Res. Not. \bf 2004:50 \rm (2004), 2703--2717.

\bibitem{Gri}
\sc A…Grigis,
\rm Estimations asymptotiques des intervalles d’instabilité pour l’équation de Hill.
\it Ann. Sci. Éc. Norm. Supér., IV. Sér \bf 20 \rm(1987), 641--672.

\bibitem{Har}
\sc E…Harrell,
\rm On the effect of the boundary conditions on the eigenvalues of ordinary differential equations.
\it Contributions to Analysis and Geometry (Baltimore, 1980), \rm 139--150, Johns Hopkins University Press, Baltimore, 1981.

\bibitem{Ho}
\sc H…Hochstadt,
\rm Estimates on the stability interval's for the Hill's equation.
\it Proc. AMS \bf 14 \rm (1963), 930--932.

\bibitem{KM-99}
\sc T…Kappeler \& B…Mityagin,
\rm Gap estimates of the spectrum of Hill's equation and action variables for KdV.
\it Trans. Amer. Math. Soc. \bf 351 \rm(1999), 619--646.

\bibitem{KM-01}
\sc T…Kappeler \& B…Mityagin,
\rm Estimates for periodic and Dirichlet eigenvalues of the Schrödinger operator.
\it SIAM J. Math. Anal. \bf 33 \rm(2001), 113--152.
 
\bibitem{MO}
\sc V…A…Mar\v{c}enko \& I…O…Ostrowsk\u{i},
\rm A charactrization of the spectrum of Hill's operator.
\it Math. USSR Sbornik \bf 97 \rm (1975), 493--554.

\bibitem{ST-96}
\sc J…J…Sansuc \& V…Tkachenko,
\rm Spectral properties of non-selfadjoint Hill's operators with smooth potentials.
\it Algebraic and Geometric Methods in Mathematical Physics (Kaciveli, 1993), \rm 371--385, Kluver, 1996.

\bibitem{Tk-01}
\sc V…Tkachenko,
\rm Characterization of Hill operators with analytic potentials.
\it Integral Equations Operator Theory, \rm to appear.

\bibitem{Tr}
\sc E…Trubowitz,
\rm The inverse problem for periodic potentials.
\it Comm. Pure Appl. Math. \bf 30 \rm (1977), 321--342.

\end{thebibliography}
\end{document}